\documentclass[12pt]{amsart}
\usepackage{amsmath,amssymb,fullpage}

\title[Contractive determinantal representations]{Contractive determinantal representations of stable polynomials
on a matrix polyball}

\author[Grinshpan]{Anatolii Grinshpan}
\author[Kaliuzhnyi-Verbovetskyi]{Dmitry~S.~Kaliuzhnyi-Verbovetskyi}
\author[Vinnikov]{Victor Vinnikov}
\author[Woerdeman]{Hugo J.~Woerdeman}
\address{Department of Mathematics \\
Drexel University\\
3141 Chestnut St.\\
Philadelphia, PA, 19104}
\email{\{tolya,dmitryk,hugo\}@math.drexel.edu}
\address{Department of Mathematics\\
Ben-Gurion University of the Negev\\
Beer-Sheva, Israel, 84105} \email{vinnikov@math.bgu.ac.il}

\thanks{AG, DK-V, HW were partially supported by NSF grant DMS-0901628.
DK-V and VV were partially supported by BSF grant 2010432.}

\subjclass{15A15, 32A10, 47N70, 14A22} \keywords{Contractive
determinantal representation;  stable polynomial; polyball; classical Cartan domain;
contractive realization; structured noncommutative multidimensional system.}

\theoremstyle{plain}

\newtheorem{thm}{Theorem}[section]
\newtheorem{cor}[thm]{Corollary}

\newtheorem{prop}[thm]{Proposition}

\numberwithin{equation}{section}

\newcommand{\beq}{\begin{equation}}
\newcommand{\eeq}{\end{equation}}

\newcommand{\mat}[2]{\ensuremath{{#1}^{#2\times #2}}}

\newcommand{\rmat}[3]{\ensuremath{{#1}}^{#2\times #3}}

\theoremstyle{remark}

\newcommand{\diag}{\operatorname{diag}}

\newcommand{\row}{\operatorname{row}}
\newcommand{\col}{\operatorname{col}}

\newcommand{\range}{\operatorname{range}}
\numberwithin{equation}{section}

\newcommand{\bbm}{\begin{bmatrix}}
\newcommand{\ebm}{\end{bmatrix}}

\newcommand{\dom}{\operatorname{dom}}
\newcommand{\edom}{\operatorname{edom}}
\newcommand{\rs}{\mathcal{R}}
\newcommand{\ls}{\mathcal{L}}
\newcommand{\free}{\mathcal{G}}
\begin{document}

\begin{abstract}
 We show that an irreducible polynomial $p$ with no zeros on
the closure of a matrix unit polyball, a.k.a. a cartesian product
of Cartan domains of type I, and such that $p(0)=1$, admits a
strictly contractive determinantal representation, i.e.,
$p=\det(I-KZ_n)$, where $n=(n_1,\ldots,n_k)$ is a $k$-tuple of
nonnegative integers, $Z_n=\bigoplus_{r=1}^k(Z^{(r)}\otimes
I_{n_r})$, $Z^{(r)}=[z^{(r)}_{ij}]$ are complex matrices, $p$ is a
polynomial in the matrix entries $z^{(r)}_{ij}$, and $K$ is a
strictly contractive matrix. This result is obtained via a
noncommutative lifting and a theorem on the singularities of
minimal noncommutative structured system realizations.
\end{abstract}

\maketitle

\section{Introduction}\label{sec:Intro}
Polynomial stability arises naturally in various problems of
Analysis and its applications such as Electrical Engineering and
Control Theory
\cite{Borcea,Wagner,BF,Kummert,Doyle,Gurvits,Scheicher,LXL}. A
polynomial $p\in\mathbb{C}[z_1,\ldots,z_d]$ is called stable with
respect to a domain $\mathcal{D}\subseteq\mathbb{C}^d$, or just
$\mathcal{D}$-stable, if it has no zeros in $\mathcal{D}$, and
strongly $\mathcal{D}$-stable if it has no zeros in the domain
closure $\overline{\mathcal{D}}$. In the case where $d=1$ and
$\mathcal{D}$ is the unit disk $\mathbb{D}=\{z\in\mathbb{C}\colon
|z|<1\}$, and $p(0)=1$, one can write
$$p=(1-a_1z)\cdots(1-a_nz)=\det (I-Kz),$$ where $a_i=1/z_i$,
$i=1,\ldots,n$, the zeros $z_i$ of $p$ are counted according to
their multiplicities, $K=\diag[a_1,\ldots, a_n]$, and $n=\deg p$.
It follows that the matrix $K$ is contractive (resp., strictly
contractive), i.e., $\|K\|\le 1$ (resp., $\|K\|<1$); here and
throughout the paper, $\|\cdot\|$ is the operator $(2,2)$ norm.

In the case where $d=2$ and $\mathcal{D}$ is the unit bidisk
$\mathbb{D}^2$, it is also true that a stable (resp., strongly
stable) polynomial $p$ has a contractive (resp., strictly
contractive) determinantal representation. It was shown in
\cite{GKVVW} (see also \cite{Kummert0,Kummert,Kummert1}) that
every $\mathbb{D}^2$-stable (resp., strongly
$\mathbb{D}^2$-stable) polynomial $p$, with $p(0)=1$, can be
represented as \begin{equation}\label{eq:polydisk-det} p=\det
(I-KZ_n),
\end{equation}
 where
$n=(n_1,n_2)\in\mathbb{Z}_+^2$ is the bi-degree of $p$,
$Z_n=\diag[z_1I_{n_1},z_2I_{n_2}]$\footnote{Here and in the rest
of the paper we use a convention that a matrix block which
involves $I_{n_i}$ is void in the case of $n_i$ equal to $0$.},
and the matrix $K$ is contractive (resp., strictly contractive).

For a higher-dimensional polydisk case,
$\mathcal{D}=\mathbb{D}^d$, $d>2$, it is in general not true that
every stable (resp., strongly stable) polynomial $p$, with
$p(0)=1$, has a determinantal representation
 \eqref{eq:polydisk-det}
 where now
$n=(n_1,\ldots,n_d)\in\mathbb{Z}_+^d$ is equal to the multi-degree
of $p$, $\deg p$, $Z_n=\diag[z_1I_{n_1},\ldots,z_dI_{n_d}]$, and
the matrix $K$ is contractive (resp., strictly contractive). Such
a representation with $n\ge\deg p$ (in the sense that
$n_i\ge\deg_ip$, $i=1,\ldots,d$, where $\deg_ip$ denotes the
$i$-th partial degree of $p$) has been constructed for some
special classes of stable polynomials in \cite{GKVW}.

The existence of a representation \eqref{eq:polydisk-det} with a
 contractive (resp., strictly contractive) matrix $K$ provides a certificate for
stability (resp., strong stability) of a polynomial $p$. Moreover,
if merely a polynomial multiple of $p$ has such a representation,
the stability (resp., strong stability) of $p$ is guaranteed. In a
recent paper of the authors \cite{GKVVW1}, the following result
has been obtained. Let
$\mathbf{P}\in\rmat{\mathbb{C}}{\ell}{m}[z_1,\ldots,z_d]$ and
$\mathbf{P}^{(r)}\in\rmat{\mathbb{C}}{\ell_r}{m_r}[z_1,\ldots,z_d]$,
$r=1,\ldots,k$, be such that
$\mathbf{P}=\bigoplus_{r=1}^k\mathbf{P}^{(r)}$, and let
\begin{equation}\label{eq:p-dom}
\mathcal{D}_\mathbf{P}=\{z\in\mathbb{C}^d\colon\|\mathbf{P}(z)\|<1\}.
\end{equation}
Under an appropriate Archimedean condition on $\mathbf{P}$, which
in particular implies the boundedness of the domain
$\mathcal{D}_\mathbf{P}$,  for every strongly
$\mathcal{D}_\mathbf{P}$-stable polynomial
$p\in\mathbb{C}[z_1,\ldots,z_d]$ there exists a polynomial
$q\in\mathbb{C}[z_1,\ldots,z_d]$ such that $pq$ has a
determinantal representation
\begin{equation}\label{eq:pq}
pq=\det(I-K\mathbf{P}_n),
\end{equation}
where $n=(n_1,\ldots,n_k)\in\mathbb{Z}_+^k$,
$\mathbf{P}_n=\bigoplus_{r=1}^k(\mathbf{P}^{(r)}\otimes I_{n_r})$,
and the matrix $K$ is strictly contractive. We note that special
cases of the domain $\mathcal{D}_\mathbf{P}$ as above include the
unit polydisk and the classical Cartan domains of type I, II, and
III, as well as the Cartesian products of such domains.

 In this paper, we construct strictly contractive determinantal representations for strongly stable polynomials
 on a Cartesian product of Cartan's domains of type I, i.e., on a
matrix unit polyball,
\begin{multline}\label{eq:polyball}
\mathbb{B}^{\ell_1\times
m_1}\times\cdots\times\mathbb{B}^{\ell_k\times
m_k}\\
=\Big\{Z=(Z^{(1)},\ldots,Z^{(k)})\in\rmat{\mathbb{C}}{\ell_1}{m_1}\times\cdots\times
\rmat{\mathbb{C}}{\ell_k}{m_k}\colon\|Z^{(r)}\|<1,\
r=1,\ldots,k\Big\}.\end{multline} In other words, in the case
where $\mathcal{D}_\mathbf{P}$ is a unit matrix polyball, i.e.,
where
 $\mathbf{P}^{(r)}=Z^{(r)}$, $r=1,\ldots,k$, no additional polynomial factor $q$ is needed to construct a strictly
contractive determinantal representation \eqref{eq:pq}, and we
have
$$p=\det(I-KZ_n),\quad Z_n=\bigoplus_{r=1}^k(Z^{(r)}\otimes
I_{n_r})$$ with some $n=(n_1,\ldots,n_k)\in\mathbb{Z}_+^k$ and
some $\sum_{r=1}^km_rn_r\times\sum_{r=1}^k\ell_rn_r$ matrix $K$
such that $\|K\|<1$. Notice that the unit polydisk $\mathbb{D}^d$
is a special case of a unit matrix polyball where $k=d$, and
$\ell_r=m_r=1$ for $r=1,\ldots,k$.

The proof of our main theorem has two components: realization
formulas for multivariable rational functions and related
techniques from multidimensional system theory, and results from a
theory of noncommutative rational functions.

The first component was a key for constructing determinantal
representations in
\cite{Kummert0,Kummert1,Kummert,GKVW,GKVVW,GKVVW1}. We recall
 (see \cite[Proposition 11]{Arov}) that every
  matrix-valued rational function that is regular and contractive on the
open unit disk ${\mathbb D}$ can be realized as
\begin{equation*}\label{real} F= D + zC(I-zA)^{-1} B,
\end{equation*} with a contractive colligation matrix $\left[\begin{smallmatrix} A & B \\ C & D
\end{smallmatrix}\right]$. In
several variables, the celebrated result of Agler \cite{Ag} gives
the existence of a realization of the form
\begin{equation*}\label{real2} F(z)= D + CZ_\mathcal{X}(I-AZ_{\mathcal{X}})^{-1} B,\qquad
Z_\mathcal{X} = \bigoplus_{i=1}^d z_i I_{{\mathcal X}_i},
 \end{equation*}
where $z=(z_1,\ldots,z_d)\in\mathbb{D}^d$ and the colligation
$\left[
\begin{smallmatrix} A & B \\ C & D
\end{smallmatrix}\right]$ is a Hilbert-space unitary operator (with $A$
acting on the orthogonal direct sum of Hilbert spaces ${\mathcal
X}_1, \ldots , {\mathcal X}_d$), for $F$ an operator-valued
function holomorphic on the unit polydisk $\mathbb{D}^d$ whose
Agler norm
$$\|F\|_\mathcal{A}=\sup_{T\in\mathcal{T}}\|F(T)\|$$ is at most 1.
Here $\mathcal{T}$ is the set of $d$-tuples $T=(T_1,\ldots,T_d)$
of commuting strict contractions on a Hilbert space. Such
functions constitute the Schur--Agler class.

Agler's result was generalized to polynomially defined domains in
\cite{AT,BB}. Given
$\mathbf{P}\in\rmat{\mathbb{C}}{\ell}{m}[z_1,\ldots,z_d]$, let
$\mathcal{D}_\mathbf{P}$ be as in \eqref{eq:p-dom} (here we can
assume that $k=1$ and $\mathbf{P}=\mathbf{P}_1$) and let
${\mathcal T}_\mathbf{P}$ be the set of $d$-tuples $T$ of
commuting bounded operators on a Hilbert space satisfying
$\|\mathbf{P}(T)\|<1$. (The case of unit polydisk $\mathbb{D}^d$
corresponds to  $\mathbf{P}=\diag[z_1,\ldots,z_d]$ and  ${\mathcal
T}_\mathbf{P}=\mathcal{T}$.) For $T \in {\mathcal T}_\mathbf{P}$,
the Taylor joint spectrum $\sigma(T)$ \cite{T1} lies in
$\mathcal{D}_\mathbf{P}$ (see \cite[Lemma 1]{AT}), and therefore
for an operator-valued function $F$ holomorphic on
$\mathcal{D}_\mathbf{P}$ one defines $F(T)$ by means of Taylor's
functional calculus \cite{T2} and
\begin{equation}\label{eq:Agler-norm} \| F \|_{{\mathcal A},{\mathbf P}}= \sup_{T
\in {\mathcal T}_{\mathbf{P}}} \| F(T) \|. \end{equation} We say that $F$
belongs to the operator-valued Schur--Agler class associated with
$\mathbf{P}$, denoted by
$\mathcal{SA}_\mathbf{P}(\mathcal{U},\mathcal{Y})$, if $F$ is
holomorphic on $\mathcal{D}_\mathbf{P}$, takes values in the space
$\mathcal{L}(\mathcal{U},\mathcal{Y})$ of bounded linear operators
from a Hilbert space $\mathcal{U}$ to a Hilbert space
$\mathcal{Y}$, and $\| F \|_{{\mathcal A},{\mathbf P}}\le 1$.

The generalization of Agler's theorem mentioned above that has
appeared first in \cite{AT} for the scalar-valued case and
extended in \cite{BB} to the operator-valued case, says that a
function $F$ belongs to the Schur--Agler class
$\mathcal{SA}_\mathbf{P}(\mathcal{U},\mathcal{Y})$ if and only if
there exist a Hilbert space $\mathcal{X}$ and a unitary
colligation
$$\begin{bmatrix} A & B \\ C & D\end{bmatrix}\colon
(\mathbb{C}^m\otimes \mathcal{X})\oplus
\mathcal{U}\to(\mathbb{C}^\ell\otimes \mathcal{X})\oplus
\mathcal{Y}$$ such that
\begin{equation}\label{eq:realization}
F(z)=D+C(\mathbf{P}(z)\otimes
I_\mathcal{X})\Big(I-A(\mathbf{P}(z)\otimes
I_\mathcal{X})\Big)^{-1}B.
\end{equation}

If the Hilbert spaces $\mathcal{U}$ and $\mathcal{Y}$ are
finite-dimensional, $F$ can be treated as a matrix-valued function
(relative to a pair of orthonormal bases for $\mathcal{U}$ and
$\mathcal{Y}$). It is natural to ask whether every rational
$\alpha\times\beta$ matrix-valued function in the Schur--Agler
class
$\mathcal{SA}_\mathbf{P}(\mathbb{C}^\beta,\mathbb{C}^\alpha)$ has
a realization \eqref{eq:realization} with a contractive
colligation matrix $\left[
\begin{smallmatrix} A & B \\ C & D
\end{smallmatrix}\right]$. This question is open for $d>1$, except for the following two cases.
The first case is when $F$ is an inner (i.e., regular on
$\mathbb{D}^d$ and taking unitary boundary values a.e. on the unit
torus $\mathbb{T}^d=\{z=(z_1,\ldots,z_d)\in\mathbb{C}^d\colon
|z_i|=1,\ i=1,\ldots,d\}$) matrix-valued Schur--Agler function on
$\mathbb{D}^d$. In this case, the colligation matrix for the
realization \eqref{eq:realization} can be chosen unitary; see
\cite{Knese2011} for the scalar-valued case, and \cite[Theorem
2.1]{BK} for the matrix-valued generalization. We notice here that
not every rational inner function is Schur--Agler; see
\cite[Example 5.1]{GKVW} for a counterexample. In the second case,
one assumes that $\mathbf{P}=\bigoplus_{r=1}^k\mathbf{P}^{(r)}$
satisfies a certain matrix-valued Archimedean condition and $F$ is
regular on the closed domain $\overline{\mathcal{D}}_\mathbf{P}$
and satisfies $\|F\|_{\mathcal{A},\mathbf{P}}<1$. Then there exists a contractive finite-dimensional realization of $F$ in the
form
\begin{equation}\label{eq:P-realiz}
F = D +C\mathbf{P}_n(I-A\mathbf{P}_n)^{-1}B,\qquad
\mathbf{P}_n=\bigoplus_{r=1}^k(\mathbf{P}^{(r)}\otimes I_{n_r}),
\end{equation}
with some $n=(n_1,\ldots,n_k)\in\mathbb{Z}_+^k$ \cite{GKVVW1}.

The second component in the proof of our main result, a theory of
noncommutative rational functions, is briefly summarized in
Section \ref{sec:Append}. Then a version of a theorem from
\cite{KV2009} on the singularities of a noncommutative rational
matrix-valued function in terms of its minimal realization, where the realization is in the form of a structured noncommutative multidimensional system, i.e., the one that is associated with a unit polyball \eqref{eq:polyball}, is proved (see \cite{BGM} for
details on structured noncommutative multidimensional systems). As a corollary,
an analogous theorem on the singularities of a commutative matrix-valued
rational function is obtained via a noncommutative lifting.

In Section \ref{sec:Detrep}, our main theorem is proved, which
establishes the existence of a strictly contractive determinantal
representation for every irreducible strongly stable polynomial on
a matrix polyball. As a corollary, in the case of the unit
polydisk $\mathbb{D}^d$, we obtain that every strongly stable
polynomial $p$ is an eventual Agler denominator, i.e., is the
denominator of a rational inner function of the Schur--Agler
class.

\section{Singularities of noncommutative rational functions and minimal structured noncommutative multidimensional
systems} \label{sec:Append}

We first give some necessary background on matrix-valued
noncommutative rational functions; see \cite{KV2009,KV2012} for
more details, and we also refer to \cite{KVV} for a general theory
of free noncommutative functions.

A matrix-valued noncommutative rational expression $R$ over a
field $\mathbb{K}$ is any expression obtained from noncommuting
indeterminates $z_1$, \ldots, $z_d$, and a constant
$1\in\mathbb{K}$ by successive elementary operations: addition,
multiplication, and inversion, forming (block) matrices, and also
matrix addition, multiplication, and inversion. E.g., an
$\alpha\times\beta$ matrix-valued noncommutative polynomial
$$R=\sum_{w\in\mathcal{G}_d\colon |w|\le L}R_wz^w$$ is a matrix-valued
noncommutative rational expression defined without using
inversions. Here $\mathcal{G}_d$ is the free monoid on $d$
generators (letters) $g_1$, \ldots, $g_d$, the coefficients $R_w$
are $\alpha\times \beta$ matrices over $\mathbb{K}$, and for an
element $w=g_{i_1}\cdots g_{i_N}\in\mathcal{G}_d$ (a word in the
alphabet $g_1$, \ldots, $g_d$) we set $z^w=z_{i_1}\cdots z_{i_N}$
and $z^\emptyset=1$, where $\emptyset$ is the unit element of
$\mathcal{G}_d$ (the empty word), and $|w|=N$ is the length of the
word $w$, in particular $|\emptyset|=0$.

 For a $d$-tuple $Z=(Z_1,\ldots,Z_d)$
of $s\times s$ matrices over $\mathbb{K}$, one can evaluate
$$R(Z)=R_s(Z)=\sum_{w\in\mathcal{G}_d}R_w\otimes Z^w,$$
where $Z^w=Z_{i_1}\cdots Z_{i_k}$ and $Z^\emptyset =I_s$.
Similarly, one can evaluate $R$ on a $d$-tuple
$X=(X_1,\ldots,X_d)$ of generic $s\times s$ matrices, i.e., on a
$d$-tuple of matrices over commuting indeterminates $(X_r)_{ij}$,
$r=1,\ldots,d$, $i,j=1,\ldots,s$. We then define evaluations
$R(Z)=R_s(Z)$ and $R(X)=R_s(X)$ whenever all the formal matrix
inversions in the expression $R$ can be replaced by matrix
inversions for matrices over $\mathbb{K}$ (resp., for generic
matrices); this defines the domain of regularity of  $R$, ${\rm
dom}_sR$, and the extended domain of regularity of $R$, ${\rm
edom}_sR$, inside the set of $d$-tuples of $s\times s$ matrices
over $\mathbb{K}$ ($d$-tuples of generic matrices). Then one
defines ${\rm dom}\,R=\coprod_{s=1}^\infty {\rm dom}_sR$ and ${\rm
edom}\,R=\coprod_{s=1}^\infty {\rm edom}_sR$. One has
$${\rm edom}_s{R}\supseteq{\rm
dom}_s{R},\quad {\rm edom\,}{R}\supseteq{\rm dom\,}{R}.$$

Two $\alpha\times \beta$ matrix-valued noncommutative rational
expressions $R_1$ and $R_2$ are called equivalent if $\dom\,
R_1\cap\dom\,R_2\neq\emptyset$ and $R_1(Z)=R_2(Z)$ for every $Z\in
\dom\, R_1\cap\dom\,R_2$. An equivalence class of $\alpha\times
\beta$ matrix-valued noncommutative rational expressions is called
an $\alpha\times \beta$ matrix-valued noncommutative rational
function. We write $R\in\mathfrak{R}$ if a matrix-valued
noncommutative rational function $\mathfrak{R}$ as an equivalence
class of matrix-valued noncommutative rational expressions
contains $R$. We define
$${\rm dom}_s\,\mathfrak{R}=\bigcap_{R\in\mathfrak{R}}{\rm
dom}_sR,\quad {\rm dom\,}\mathfrak{R}=\coprod_{s=1}^\infty {\rm
dom}_s\mathfrak{R}.$$ Next, we observe that if $R_1$ and $R_2$ are
equivalent, then their evaluations on generic matrices give rise
to the same  $\alpha\times \beta$ matrix-valued commutative
rational function, so ${\rm edom}_sR_1={\rm edom}_sR_2$ for every
$s$. Therefore, we can define $${\rm edom}_s\mathfrak{R}={\rm
edom}_sR,\quad {\rm edom\,}\mathfrak{R}={\rm edom}\,R$$ for any
$R\in\mathfrak{R}$. Clearly, we have
$${\rm edom}_s\mathfrak{R}\supseteq{\rm
dom}_s\mathfrak{R},\quad {\rm edom\,}\mathfrak{R}\supseteq{\rm
dom\,}\mathfrak{R}.$$

In \cite{KV2009}, the left and right backward shift operators
$\mathcal{L}_j$ and $\mathcal{R}_j$, $j=1,\ldots,d$, were defined
for matrix-valued noncommutative rational expressions. It was
shown that if $R_1$ and $R_2$ are equivalent, then so are
$\mathcal{L}_j(R_1)$ and $\mathcal{L}_j(R_2)$ (resp.,
$\mathcal{R}_j(R_1)$ and $\mathcal{R}_j(R_2)$). Therefore, these
definitions can be extended to matrix-valued noncommutative rational functions.
One defines $${\rm dom}\,\mathcal{L}_j(R)={\rm
dom}\,\mathcal{R}_j(R)={\rm dom}\,R,$$ however we have
$${\rm
edom}\,\mathcal{L}_j(R)\supseteq {\rm edom}\,R,\quad {\rm
edom}\,\mathcal{R}_j(R)\supseteq {\rm edom}\,R,$$ and therefore
$${\rm
edom}\,\mathcal{L}_j(\mathfrak{R})\supseteq {\rm
edom}\,\mathfrak{R},\quad {\rm
edom}\,\mathcal{R}_j(\mathfrak{R})\supseteq {\rm
edom}\,\mathfrak{R}.$$ We will not need the general definitions of
the left and right backward shifts here. It suffices for us to use
the fact that every matrix-valued noncommutative rational
function $\mathfrak{R}$ which is regular at $0$, i.e., such that
$0\in{\rm dom}_1\mathfrak{R}$, has a formal power series expansion
$$\mathfrak{R}\sim\sum_{w\in\mathcal{G}_d}\mathfrak{R}_wz^w,$$
whose evaluation on $s\times s$ matrices is convergent in some
neighborhood of zero for each $s$, and that
$$\mathcal{L}_j\mathfrak{R}\sim\sum_{w\in\mathcal{G}_d}\mathfrak{R}_{g_jw}z^w,\quad
\mathcal{R}_j\mathfrak{R}\sim\sum_{w\in\mathcal{G}_d}\mathfrak{R}_{wg_j}z^w.$$

The following theorem is a structured-system analogue of
\cite[Theorem 3.1]{KV2009}; for details on
structured noncommutative multidimensional systems, see \cite{BGM}. We note that we are not using here a bipartite-graph formalism adopted in \cite{BGM} for system evolutions and, as a consequence, for the definitions of controllability and observability. Instead, we use more direct block-matrix notations. The diligent reader can easily find the one-to-one correspondence between the two formalisms.

\begin{thm}\label{thm:structured-sing}
Let $\mathfrak{R}$ be an $\alpha\times\beta$ matrix-valued
noncommutative rational function over a field $\mathbb{K}$
represented by the expression \begin{equation}\label{eq:R} R=
D+Cz_n(I-Az_n)^{-1}B,\end{equation} where
$n=(n_1,\ldots,n_k)\in\mathbb{Z}_+^k$,
$z_n=\bigoplus_{r=1}^k(z^{(r)}\otimes I_{n_r})$,
$z^{(r)}=[z^{(r)}_{ij}]$ is a $\ell_r\times m_r$ matrix whose
entries $z^{(r)}_{ij}$ are noncommuting
indeterminates, $\left[\begin{smallmatrix} A & B\\
C & D
\end{smallmatrix}\right]\in\mathbb{K}^{(\sum_{r=1}^km_rn_r+\alpha)\times (\sum_{r=1}^k\ell_r
n_r+\beta)}$ is a block matrix whose blocks $A,B,C$  have further
block decompositions\footnote{Here, similarly to the convention we
made in a footnote on the front page of the paper, we assume that
a matrix block is void if the number of its rows/columns is $0$.}:
$A=[A^{(rr')}]_{r,r'=1,\ldots,k}$ with blocks
$A^{(rr')}\in\mathbb{K}^{m_r\times\ell_{r'}}\otimes\mathbb{K}^{n_r\times
n_{r'}} \cong(\mathbb{K}^{n_r\times
n_{r'}})^{m_r\times\ell_{r'}}$, so that for $j=1,\ldots,m_r$,
$i=1,\ldots,\ell_{r'}$ one has
$A^{(rr')}_{ji}\in\mathbb{K}^{n_r\times n_{r'}}$;
$B=\col_{r=1,\ldots, k}[B^{(r)}]$ with blocks
$B^{(r)}\in\mathbb{K}^{m_r\times
1}\otimes\mathbb{K}^{n_r\times\beta}\cong(\mathbb{K}^{n_r\times\beta})^{m_r\times
1}$, so that for $j=1,\ldots,m_r$ one has
$B_j^{(r)}\in\mathbb{K}^{n_r\times\beta}$;
 $C=\row_{i=1,\ldots,k}[C^{(r)}]$
with blocks $C^{(r)}\in\mathbb{K}^{1\times
\ell_r}\otimes\mathbb{K}^{\alpha\times
n_r}\cong(\mathbb{K}^{\alpha\times n_r})^{1 \times\ell_r}$, so
that for $i=1,\ldots,\ell_r$ one has
$C_i^{(r)}\in\mathbb{K}^{\alpha\times n_r}$; and
$D\in\mathbb{K}^{\alpha\times\beta}$. Assume that the realization
$R$ of $\mathfrak{R}$ as in \eqref{eq:R} is minimal, or
equivalently, controllable, i.e., for each $r_0\in\{1,\ldots,k\}$
and $j_0\in\{1,\ldots,m_{r_0}\}$ one has
\begin{equation}\label{controll}
\sum\limits_{N\in\mathbb{N},\,\gamma\in\{1,\ldots,N\},\,i_\gamma\in\{1,\ldots,\ell_{r_\gamma}\},\,
j_\gamma\in\{1,\ldots,m_{r_\gamma}\}}\range (A^{(r_0r_1)}_{j_0i_1}\cdots
A^{(r_{N-1}r_N)}_{j_{N-1}i_N}B^{(r_N)}_{j_N})=\mathbb{K}^{n_{r_0}},
\end{equation}
 and observable, i.e., for each $r_0\in\{1,\ldots,k\}$ and $i_0\in\{1,\ldots,\ell_{r_0}\}$ one has
\begin{equation}\label{observ}
\bigcap_{N\in\mathbb{N},\,\gamma\in\{1,\ldots,N\},\,i_\gamma\in\{1,\ldots,\ell_{r_\gamma}\},\,
j_\gamma\in\{1,\ldots,m_{r_\gamma}\} }\ker(C^{(r_N)}_{i_N}A^{(r_Nr_{N-1})}_{j_Ni_{N-1}}\cdots
A^{(r_1r_0)}_{j_1i_0})=\{0\}.
\end{equation} Then
\begin{multline}\label{eq:edom}
{\rm edom\,}\mathfrak{R}={\rm
dom\,}\mathfrak{R}=\coprod_{s=1}^\infty\Big\{Z=(Z^{(1)},\ldots,Z^{(k)})\in(\mathbb{K}^{s\times
s})^{\ell_1\times m_1}\times\cdots\times(\mathbb{K}^{s\times s})^{\ell_k\times m_k}\\
\cong(\mathbb{K}^{\ell_1\times
m_1}\times\cdots\times\mathbb{K}^{\ell_k\times m_k})\otimes\mathbb{K}^{s\times s}\colon\det(I-A\odot Z)\neq 0\Big\},
\end{multline}
where $A\odot Z\in\mathbb{K}^{ \sum_{r=1}^km_rn_rs\ \times \sum_{r=1}^km_rn_rs}$ is a  block
$\sum_{r=1}^km_r\times \sum_{r=1}^km_r$ matrix with blocks
$$(A\odot Z)^{(rr')}_{ij}=\sum_{\kappa=1}^{\ell_{r'}} A^{(rr')}_{i\kappa}\otimes Z^{(r')}_{\kappa j}\in\mathbb{K}^{n_r\times n_{r'}}\otimes\mathbb{K}^{s\times s}\cong
\mathbb{K}^{ n_rs\times n_{r'}s},$$
$i=1,\ldots,m_r$, $j=1,\ldots,m_{r'}$.
\end{thm}

\begin{proof} It is clear that the inclusion ``$\supseteq$" holds in both the equalities
in \eqref{eq:edom}.

Conversely, let $Z\in \edom_s\mathfrak{R}$ for some $s\in\mathbb{N}$. We will show
that $\det(I -A\odot Z)\neq 0.$ Let
$$\ls=(\ls^{(r)}_{ij})_{r=1,\ldots,k,\,i=1,\ldots,\ell_r,\,j=1,\ldots,m_r},\quad
\rs=(\rs^{(r)}_{ij})_{i=1,\ldots,\ell_r,\,j=1,\ldots,m_r}$$ be the $d$-tuples of left and right backward shifts, where $d=\sum_{r=1}^k\ell_rm_r$. For a word
$w=g^{(r_0)}_{i_0j_0}\cdots g^{(r_N)}_{i_Nj_N}\in\free_d$, we set
$\ls^w=\ls^{(r_0)}_{i_0j_0}\cdots\ls^{(r_N)}_{i_Nj_N}$,
$\rs^w=\rs^{(r_0)}_{i_0j_0}\cdots\rs^{(r_N)}_{i_Nj_N}$. Then for any
$w=g^{(r_0)}_{i_0j_0}\cdots g^{(r_N)}_{i_Nj_N}\in\free_d$ with $N\ge 1$
we obtain \begin{multline*} \rs^w(R)=\rs^w\Big(Cz_n(I-Az_n)^{-1}B\Big)\\
=\rs^w\Big(C(I-z_nA)^{-1}z_nB\Big)
=C(I-z_nA)^{-1}\Big(\mathbf{e}^{(r_0)}_{i_0}\otimes
A^{(r_0r_1)}_{j_0i_1}\cdots
A^{(r_{N-1}r_N)}_{j_{N-1}i_N}B^{(r_N)}_{j_N}\Big),\end{multline*} where $\mathbf{e}^{(r_0)}_{i_0}$ is
the $\Big(\sum_{r=1}^{r_0-1}\ell_r+i_0\Big)$-th standard basis vector of $\mathbb{K}^{\sum_{r=1}^k\ell_r}$. Since
$Z\in\edom_s\mathfrak{R}$, we have
$Z\in\edom_s\rs^w\mathfrak{R}=\edom_s\rs^w(R)$. Therefore, the
$\alpha s\times\beta s$ matrix-valued rational function
$$(C\otimes I_s)(I-X\odot_{\rm op} A)^{-1}\Big(\mathbf{e}^{(r_0)}_{i_0}\otimes A^{(r_0r_1)}_{j_0i_1}\cdots
A^{(r_{N-1}r_N)}_{j_{N-1}i_N}B^{(r_N)}_{j_N}\otimes I_s\Big)$$ in the commuting variables
$(X^{(r)}_{ij})_{\mu\nu}$, $r=1,\ldots,k$, $i=1,\ldots,\ell_r$, $j=1,\ldots,m_r$,
$\mu,\nu=1,\ldots,s$,
  is regular at $X=Z$, where $X\odot_{\rm op}
  A\in\mat{\mathbb{K}}{\sum_{r=1}^k\ell_rn_rs}$ is a block $\sum_{r=1}^k\ell_r\times\sum_{r=1}^k\ell_r$ matrix with blocks
  $$(X\odot_{\rm op} A)^{(rr')}_{ij}=\sum_{\kappa=1}^{m_r}A^{(rr')}_{\kappa j}\otimes
  X^{(r)}_{i\kappa}\in\rmat{\mathbb{K}}{n_r}{n_{r'}}\otimes\mat{\mathbb{K}}{s}\cong\rmat{\mathbb{K}}{n_rs}{n_{r'}s}.$$
 The controllability assumption implies that the $\alpha s\times \sum_{r=1}^k\ell_r n_rs$ matrix-valued
rational function $(C\otimes I_s)(I-X\odot_{\rm op} A)^{-1}$ is
regular at $X=Z$. Therefore, the $\alpha s\times \sum_{r=1}^km_rn_rs$
matrix-valued rational function $$(C\otimes I_s)(I-X\odot_{\rm op}
A)^{-1}(X\odot_{\rm op}I_{\sum_{r=1}^km_rn_r})=(C\odot X)(I-A\odot X)^{-1}$$ is regular at
$X=Z$. In other words, $Z\in \edom_sR'$ where
$$R'=Cz_n(I-Az_n)^{-1}$$ is an $\alpha\times \sum_{r=1}^km_rn_r$
matrix-valued noncommutative rational expression.

Next, for any $w=g^{(r_0)}_{i_0j_0}\cdots g^{(r_n)}_{i_Nj_N}\in\free_d$
with $N\ge 1$ we set $w^\top=g^{(r_N)}_{i_Nj_N}\cdots g^{(r_0)}_{i_0j_0}$. Then
we have
$$\ls^{w^\top}(R')=\Big({(\mathbf{f}^{(r_0)}_{j_0})}^\top\otimes C^{(r_N)}_{i_N}A^{(r_Nr_{N-1})}_{j_Ni_{N-1}}\cdots
A^{(r_1r_0)}_{j_1i_0}\Big) (I-Az_n)^{-1},$$ where
$\mathbf{f}^{(r_0)}_{j_0}$ is the $\Big(\sum_{r=1}^{r_0-1}m_r+j_0\Big)$-th standard basis vector of
$\mathbb{K}^{\sum_{r=1}^km_r}$. Since $Z\in\edom_sR'$, we have
$Z\in\edom_s\ls^{w^\top}(R')$. Therefore, the $\alpha s\times \sum_{r=1}^km_rn_rs$
matrix-valued rational function
$$\Big({(\mathbf{f}^{(r_0)}_{j_0})}^\top\otimes C^{(r_N)}_{i_N}A^{(r_Nr_{N-1})}_{j_Ni_{N-1}}\cdots
A^{(r_1r_0)}_{j_1i_0}\otimes I_s\Big)(I-A\odot X)^{-1}
$$ in the commuting variables
$(X^{(r)}_{ij})_{\mu\nu}$, $r=1,\ldots,k$, $i=1,\ldots,\ell_r$, $j=1,\ldots,m_r$,
$\mu,\nu=1,\ldots,s$,
  is regular at $X=Z$.
 The observability assumption implies that
the $\sum_{r=1}^km_rn_rs\times\sum_{r=1}^km_rn_rs$ matrix-valued rational function
$$ (I-A\odot X)^{-1}
$$ is regular at $X=Z$. Then so is the rational function $$ \det(I-A\odot
  X)^{-1}=(\det(I-A\odot
  X))^{-1},
$$
i.e., $\det(I-A\odot
  Z)\neq 0$, as required.
\end{proof}

\begin{cor}\label{cor:structured-sing}
The variety of singularities of an $\alpha\times \beta$
matrix-valued rational function $f$ which can be represented as a
restriction $R_1$ of
 an $\alpha\times \beta$
matrix-valued noncommutative rational expression $R$ of the form
\eqref{eq:R} satisfying the assumptions of Theorem
\ref{thm:structured-sing} (i.e., which is obtained from $R$ by replacing
the noncommuting indeterminates $z^{(r)}_{ij}$ by the commuting ones),
is given by
\begin{equation*}
\Big\{Z=(Z^{(1)},\ldots,Z^{(k)})\in\mathbb{K}^{\ell_1\times m_1}\times\cdots\times \mathbb{K}^{\ell_k\times m_k}\colon\det (I -AZ_n)=0\Big\},
\end{equation*}
where $Z_n=\bigoplus_{r=1}^k(Z^{(r)}\otimes I_{n_r})$.
\end{cor}

\begin{proof}
Clearly, the variety of singularities of $f$ coincides with
$$\Big(\mathbb{K}^{\ell_1 \times m_1}\times\cdots\times \mathbb{K}^{\ell_k \times m_k}\Big)\setminus\edom_1 R.$$ The result then
follows from Theorem \ref{thm:structured-sing}.
\end{proof}

We will also need to make use of the inverse of a noncommutative rational function, and of the fact that the minimality of a realization carries over to the corresponding realization of the inverse. We recall from
\cite[Section 4]{BGM} that if $\mathfrak{R}$ is an $\alpha\times\alpha$ matrix-valued
noncommutative rational function over a field $\mathbb{K}$
represented by the noncommutative rational expression \eqref{eq:R} with $D$ invertible, then its inverse exists and has a realization
\begin{equation}\label{inv} R^{-1}=
D^\times +C^\times z_n(I-A^\times z_n)^{-1}B^\times,\end{equation} where
\begin{equation}\label{invreal} \begin{bmatrix} A^\times & B^\times \\
C^\times & D^\times
\end{bmatrix} = \begin{bmatrix} A-BD^{-1}C & BD^{-1} \\
-D^{-1}C & D^{-1}
\end{bmatrix}.
\end{equation}

\begin{prop}\label{prop:inv} Assume that the realization
of $R$ in \eqref{eq:R} is minimal and that $D$ is invertible. Then the realization of $R^{-1}$ given via \eqref{inv} and \eqref{invreal} is also minimal. \end{prop}

\begin{proof}
It suffices to verify the controllability and observability for the realization of $R^{-1}$. Notice that the blocks of $A^\times$ and $B^\times$ are $A^{\times(rr')}_{ji} = A^{(rr')}_{ji}-B^{(r)}_j D^{-1} C^{(r')}_i$ and $B^{\times(r)}_j = B^{(r)}_j D^{-1} $, respectively. Thus, for the controllability, we need to check that for each $r_0\in\{1,\ldots,N\}$ and $j_0\in\{1,\ldots,m_{r_0}\}$ one has
\begin{equation}\label{controll2} \sum\limits_{N\in\mathbb{N},\,\gamma\in\{1,\ldots,N\},\,i_\gamma\in\{1,\ldots,\ell_{r_\gamma}\},\
j_\gamma\in\{1,\ldots,m_{r_\gamma}\}}\range (A^{\times(r_0r_1)}_{j_0i_1}\cdots
A^{\times(r_{N-1}r_N)}_{j_{N-1}i_N}B^{\times(r_N)}_{j_N})=\mathbb{K}^{n_{r_0}}.\end{equation}
Clearly, $ \range B^{\times(r)}_j = \range B^{(r)}_j D^{-1} = \range B^{(r)}_j$ for all $r=1,\ldots,k$ and $j=1,\ldots , m_r$. Next,
\begin{multline*}
\range (A^{\times(r_0r_1)}_{j_0i_1} B^{\times(r_1)}_{j_1}) + \range B^{\times(r_0)}_{j_0}\\ =  \range \Big((A^{(r_0r_1)}_{j_0i_1} -B^{(r_0)}_{j_0} D^{-1} C^{(r_1)}_{i_1}) B^{(r_1)}_{j_1}D^{-1}\Big) + \range B^{(r_0)}_{j_0}D^{-1}\\
  =  \range(A^{(r_0r_1)}_{j_0i_1} B^{(r_1)}_{j_1}) + \range B^{(r_0)}_{j_0}.
    \end{multline*}
Continuing this way one sees that the left hand sides of \eqref{controll} and \eqref{controll2} are the same, and thus \eqref{controll2} follows from \eqref{controll}. In a similar way, one shows the observability.
\end{proof}

\section{Contractive determinantal depresentations of stable polynomials on a matrix polyball}\label{sec:Detrep}

The main result of the paper is the following.

\begin{thm}\label{thm:polyball} Let $p$ be an irreducible polynomial in the commuting indeterminates $z^{(r)}_{ij}$, $r=1,\ldots,k$, $i=1,\ldots,\ell_r$, $j=1,\ldots,m_r$,
with  $p(0) = 1$,  which is strongly stable with respect to the matrix polyball
$\rmat{\mathbb{B}}{\ell_1}{m_1}\times\cdots\times\rmat{\mathbb{B}}{\ell_k}{m_k}$. Then there exist
 $n=(n_1,\ldots,n_k)\in{\mathbb Z}_+^k$ and a strict contraction $K \in \rmat{\mathbb{C}}{\sum_{r=1}^km_rn_r}{\sum_{r=1}^k\ell_rn_r}$ so
that \begin{equation}\label{eq:contr-det-repr}
p= \det (I - KZ_n),\end{equation}
where $Z_n=\bigoplus_{r=1}^k(Z^{(r)}\otimes I_{n_r})$ and $Z^{(r)}=[z^{(r)}_{ij}]\in\rmat{\mathbb{C}}{\ell_r}{m_r}$.
        \end{thm}

    \begin{proof}
Since $p$ has no zeros in the closed unit polyball $\rmat{\overline{\mathbb{B}}}{\ell_1}{m_1}\times\cdots\times\rmat{\overline{\mathbb{B}}}{\ell_k}{m_k}$, we have that $p$
has no zeros in $\rho\rmat{\overline{\mathbb{B}}}{\ell_1}{m_1}\times\cdots\times\rho\rmat{\overline{\mathbb{B}}}{\ell_k}{m_k}$
for some $\rho
>1$ sufficiently close to 1. Thus the rational function
$g=1/p$ is regular on
 $\rho\rmat{\overline{\mathbb{B}}}{\ell_1}{m_1}\times\cdots\times\rho\rmat{\overline{\mathbb{B}}}{\ell_k}{m_k}$, and the
rational function $g_\rho$ defined by $g_\rho(z)=g(\rho z)$ is
regular on $\rmat{\overline{\mathbb{B}}}{\ell_1}{m_1}\times\cdots\times\rmat{\overline{\mathbb{B}}}{\ell_k}{m_k}$. By
\cite[Lemma 3.3]{GKVVW1}, $\| g_\rho \|_{{\mathcal A}, Z} <
\infty$, where the corresponding Agler norm $\|\cdot\|_{{\mathcal A},Z}=\|\cdot\|_{{\mathcal A},P}$ is defined as in \eqref{eq:Agler-norm} with $\mathbf{P}=Z=\bigoplus_{r=1}^kZ^{(r)}$. Thus we
 can find a constant $c>0$ so that  $\|cg_\rho\|_{{\mathcal A}, Z} <1$.
By \cite[Theorem 3.4]{GKVVW1} applied
 to $F=cg_\rho$, we obtain a $n=(n_1,\ldots,n_k)\in\mathbb{Z}_+^k$ and a
 contractive colligation matrix $\left[\begin{smallmatrix} A & B \\ C & D \end{smallmatrix}\right]$ of size $(\sum_{r=1}^k{m_rn_r}+1)\times(\sum_{r=1}^k\ell_rn_r+1)$ such that
$$cg_\rho=D+CZ_n(I-AZ_n)^{-1}B,\quad Z_n=\bigoplus_{r=1}^k(Z^{(r)}\otimes I_{n_r}).$$
 Therefore $$
cg=D+C(\rho^{-1}Z_n)(I-A(\rho^{-1}Z_n))^{-1}B
=D+\rho^{-1}CZ_n(I-\rho^{-1}AZ_n)^{-1}B.
$$
 Then
we lift the rational function $cg$ to a noncommutative rational
expression using the same realization formula,
$$D+C'z_n(I-A'z_n))^{-1}B,$$
now with $z_n=\bigoplus_{r=1}^k(z^{(r)}\otimes I_{n_r})$ and the entries $z^{(r)}_{ij}$ of
 matrices $z^{(r)}$ being noncommuting
indeterminates, $r=1,\ldots,k$, $i=1,\ldots,\ell_r$, $j=1,\ldots,m_r$, and
$A'=\rho^{-1}A$, $C'=\rho^{-1}C$ (cf. \eqref{eq:R}). Notice that the
colligation matrix
 $\left[\begin{smallmatrix} A' & B
\\ C' & D
\end{smallmatrix}\right]$ is contractive, with $\|A'\|<1$ and $\|C'\|<1$.
 Compressing the underlying noncommutative structured system
 to a minimal one
 (see \cite[Theorem 7.1]{BGM}), we obtain a noncommutative rational expression
 $$R= D_{\rm min}+C_{\rm min}z_{n_{\rm min}} \Big(I-A_{\rm min}z_{n_{\rm
min}}\Big)^{-1}B_{\rm min},$$ whose colligation matrix
 $\left[\begin{smallmatrix} A_{\rm min} & B_{\rm min}
\\ C_{\rm min} & D_{\rm min}
\end{smallmatrix}\right]$ is still contractive and such that $\|A_{\rm min}\|<1$ and $\|C_{\rm min}\|<1$.
 By Proposition \ref{prop:inv}, we also obtain a minimal noncommutative structured system realization of $R^{-1}$,
$$R^{-1}=D^\times_{\rm min}+C^\times_{\rm min}z_{n_{\rm min}} \Big(I-A^\times_{\rm min}z_{n_{\rm
min}}\Big)^{-1}B^\times_{\rm min},$$
with the colligation matrix
\begin{equation*} \begin{bmatrix} A^\times_{\rm min} & B^\times_{\rm min} \\
C^\times_{\rm min} & D^\times_{\rm min}
\end{bmatrix} = \begin{bmatrix} A_{\rm min}-B_{\rm min}D_{\rm min}^{-1}C_{\rm min} & B_{\rm min}D_{\rm min}^{-1} \\
-D_{\rm min}^{-1}C_{\rm min} & D_{\rm min}^{-1}
\end{bmatrix}
\end{equation*}
(cf., \eqref{inv}--\eqref{invreal}).
Applying  Theorem \ref{thm:structured-sing} to $R^{-1}$ and Corollary
\ref{cor:structured-sing} to $R^{-1}_1=p/c$, we obtain that the singularity
set of the polynomial $p/c$ (which is the empty set), agrees
with
$$ \Big\{ Z=(Z^{(1)},\ldots,Z^{(k)})\in\rmat{\mathbb{C}}{\ell_{r_1}}{m_{r_1}}
\times\cdots\times\rmat{\mathbb{C}}{\ell_{r_k}}{m_{r_k}}  \colon
\det(I-A^\times_{\rm min}Z_{n_{\rm min}})
 = 0\Big\}, $$
which is possible only if $\det(I-A^\times_{\rm min}Z_{n_{\rm
min}})\equiv 1$.

Next, from the following two factorizations,
\begin{multline*}
\begin{bmatrix}
I-A_{\rm min}Z_{n_{\rm
min}} & B_{\rm min}\\
-C_{\rm min}Z_{n_{\rm min}} & D_{\rm min}
\end{bmatrix}\\
=\begin{bmatrix}
I & 0\\
-C_{\rm min}Z_{n_{\rm min}}(I-A_{\rm min}Z_{n_{\rm min}})^{-1}  &
I
\end{bmatrix}\begin{bmatrix}
I-A_{\rm min}Z_{n_{\rm
min}} & 0\\
0 & c/p
\end{bmatrix}\begin{bmatrix}
I & (I-A_{\rm min}Z_{n_{\rm
min}})^{-1} B\\
0 & I
\end{bmatrix} \\
=\begin{bmatrix}
I & B^\times_{\rm min}\\
0 & I
\end{bmatrix}\begin{bmatrix}
I-A^\times_{\rm min} Z_{n_{\rm
min}} & 0\\
0 & D_{\rm min}
\end{bmatrix}\begin{bmatrix}
I & 0\\
C^\times_{\rm min}Z_{n_{\rm min}} & I
\end{bmatrix},
\end{multline*}
we obtain that \begin{multline*}
\det\begin{bmatrix} I-A_{\rm
min}Z_{n_{\rm
min}} & B_{\rm min}\\
-C_{\rm min}Z_{n_{\rm min}} & D_{\rm min}
\end{bmatrix}=\frac{c}{p}\det(I-A_{\rm min}Z_{n_{\rm
min}})=D_{\rm min}\det(I-A^\times_{\rm min} Z_{n_{\rm
min}})\\
=D_{\rm min}=\frac{c}{p(0)}=c. \end{multline*} Therefore,
$p=\det(I-A_{\rm min}Z_{n_{\rm min}})$. Since $A_{\rm min}$ is a
strict contraction, we obtain that \eqref{eq:contr-det-repr} is
true with $K=A_{\rm min}$ and $n_{\rm min}$ in the place of $n$.
\end{proof}

\begin{cor}
Every strongly $\mathbb{D}^d$-stable polynomial $p$ is an eventual
Agler denominator, i.e., there exists
$n=(n_1,\ldots,n_d)\in\mathbb{Z}^d_+$ such that the rational inner
function
\begin{equation}\label{eq:inner}
\frac{z^n\bar{p}(1/z)}{p(z)}
\end{equation}
 is in the Schur--Agler class. Here
for $z=(z_1,\ldots,z_d)$ we set $1/z=(1/z_1,\ldots,1/z_d)$,
$\bar{p}(z)=\overline{p(\bar{z}_1,\ldots,\bar{z}_d)}$, and
$z^n=z_1^{n_1}\cdots z_d^{n_d}$.
\end{cor}
\begin{proof}
By Theorem \ref{thm:polyball} applied to the polydisk case, $p$
has a strictly contractive determinantal representation
\eqref{eq:contr-det-repr}. By \cite[Theorem 5.2]{GKVW}, $p$ is an
eventual Agler denominator. Moreover, $n$ in \eqref{eq:inner} can
be chosen the same as in a (not necessarily strictly) contractive
determinantal representation \eqref{eq:contr-det-repr} for $p$.
\end{proof}

\end{document}